\documentclass[a4paper]{amsart}

\usepackage{latexsym, amssymb,amsmath, amsthm,amsfonts}
\usepackage{tikz}
\tikzset{My Line Style/.style={smooth}}
\usetikzlibrary{shapes,arrows}

\newcommand{\kk}{\Bbbk}

\usepackage{mathtools}

\def\SL{\operatorname{SL}}

\def\SL2{\operatorname{SL}_{2}(\kk)}

\def\GL2{\operatorname{GL}_2}

\def\INVSL2{$\kk[V]^{operatorname{SL}_{2}(\kk)}$}
\def\INVSO2{$\kk[V]^{operatorname{SO}_{2}(\kk)}$}
\def\INVGL2{$\kk[V]^{operatorname{GL}_{2}(\kk)}$}

\def\Hom{\operatorname{Hom}}
\def\GL{\operatorname{GL}}
\def\SL{\operatorname{SL}}

\def\N{\mathbb{N}}

\newtheorem{Lemma}{Lemma}[section]
\newtheorem{Theorem}[Lemma]{Theorem}
\newtheorem{Corollary}[Lemma]{Corollary}

\newtheorem{Proposition}[Lemma]{Proposition}
\newtheorem{Conj}[Lemma]{Conjecture}

\theoremstyle{definition}

\theoremstyle{remark}
  
\newtheorem{eg}[Lemma]{Example}

\newtheoremstyle{Acknowledgments}
  {}
    {}
     {}
     {}
    {\bfseries}
    {}
     {.5em}
     {\thmname{#1}\thmnumber{ }\thmnote{ (#3)}}
\theoremstyle{Acknowledgments}
\newtheorem{ack}{Acknowledgments.}

\title{Some Formulae Relating Modular Representations of Elementary Abelian $p$-groups}
\author{Jonathan Elmer and Kazal Kadr}

\begin{document}
\maketitle

\begin{abstract}
 Let $p>0$ be a prime, $\kk$ a field of characteristic $p$ and $G$ an elementary abelian $p$-group of order $q=p^n$. Let $W$ be an indecomposable $\kk G$-module of dimension 2 and define $V_i = S^{i-1}(W)^*$ for each $i=1, \ldots, q$. We show that $V_2 \otimes V_i \cong V_{i+1} \oplus V_{i-1}$ provided $i$ is not divisible by $p$, and that $V_2 \otimes V_p$ is indecomposable provided $n>1$. Our results generalise results of Almkvist and Fossum \cite{AlmkvistFossum} for representations of cyclic groups of order $p$. We show how our results give formulae for the direct sum decompostion of $V_i \otimes V_j$ for all $i<p$ and $j<q$ modulo summands projective to $M:= \bigoplus_{r=0}^{p^{n-1}}V_{rp}$, and conjecture that these formulae extend to the case $i<q$ and $j<q$. We provide some evidence for our conjecture. 
\end{abstract}

\section{Introduction} Let $p>0$ be a prime, $\kk$ a field of characteristic $p$ and $G$ an elementary abelian $p$-group of order $q=p^n$. The representation theory of $G$ over $\kk$ is well-known to be wild unless $n=1$ or $q=4$. Outside of these cases it is believed we have no hope of classifying indecomposable representations of $G$ over $\kk$. Nevertheless, it is possible to describe relationships which exist between certain subclasses of the class of all such representations. The purpose of this paper is to describe some of these relationships, focussing on symmetric powers and tensor products.

We begin by fixing enough notation to state our results; precise definitions of the terms involved will be given in section 2. Let $W$ be an indecomposable $\kk G$-module with dimension 2. As $G$ is a $p$-group, there exists a basis of $W$ with respect to which the action of $G$ is upper triangular. More precisely, the action of $G$ on $W$ with respect to this basis is given by left-multiplication with the matrices
\[\left\{\begin{pmatrix} 1 & \omega \\ 0 & 1 \end{pmatrix}: \omega \in E_W\right\}\] where $E_W$ is a subgroup of $\kk$ under addition. If the action of $G$ is faithful we have $G \cong E_W$. Two such $\kk G$-modules $W, W'$ are isomorphic if and only if $E_W = \alpha E_W'$ for some $\alpha \in \kk$.

Fix an indecomposable faithful 2-dimensional $\kk G$-module $W$ and set $V_i = S^{i-1}(W)^*$ where $S$ denotes symmetric powers and $*$ the dual (or contragedient) module, and $i\leq q$. Our main result is the following:

\begin{Theorem}\label{tensorformula}\

\begin{itemize}
\item[(a)]$V_2 \otimes V_i \cong V_{i+1} \oplus V_{i-1}$ for all $i<q$ with $p \not | \ i$;

\item[(b)] $V_2 \otimes V_p$ is indecomposable if $n>1$.
\end{itemize}
\end{Theorem}
 
Let $M:= \sum_{r=1}^{p^{n-1}}V_{rp}$. We can use Theorem \ref{tensorformula} to describe certain other tensor products $V_i \otimes V_j$ up to the addition of modules which are projective relative to $M$, i.e. direct summands of a tensor multiple of $M$. Our most complete results are:

\begin{Corollary}\label{newtensorformula}
Let $i<p$ and $j<q$. Assume $i+j \leq q$. Write $j = rp+ j'$ and assume that $i+j' \leq p$. Then we have
\begin{equation}\label{newtensorbyVj} V_i \otimes V_j \cong_M \bigoplus_{l=1}^{\min(i,j')} V_{j+i-(2l-1)}\end{equation}
\end{Corollary}

\begin{Corollary}\label{newtensorbyVp-j}
Let $i<p$ and $j<q$. Assume $i+j \leq q$. Write $j = rp+ j'$ and assume that $i+j' \geq p$. Then we have
\[V_{p-i} \otimes V_{j} \cong_M V_i \otimes V_{pr+p-j'}.\] where $M$ is as in Corollary \ref{newtensorformula}.
\end{Corollary}

In the above two Corollaries, $\cong_M$ means isomorphism in the stable module category relative to $M$.

\begin{ack}
This paper contains the most substantial results of the second author's PhD thesis. The second author would like to thank Middlesex University for their support. Both authors thank Brendan Masterson and Emilio Pierro for useful conversations in the university's algebra reading group and elsewhere.
\end{ack}

\section{Background}
\subsection{Exterior and Symmetric Powers}

In this section let $G$ be a group, $\kk$ a field and $V, W$ finite dimensional vector spaces on which $G$ acts linearly. Let $x_1, \ldots, x_n$ be a basis for $V$ and let $y_1, \ldots, y_m$ be a basis for $W$. The tensor product $V \otimes W$ is the vector space with basis consisting of the formal symbols $x_i \otimes y_j: i=1, \ldots,n$, $j=1, \ldots, m$, and $G$ acts on $V \otimes W$ via linear extension of $g(x_i \otimes y_j) = gx_i \otimes gy_j$. One can take the tensor product of $V$ with itself to obtain the tensor square $T^2(V) = V \otimes V$, and iterate this procedure to obtain a graded algebra $T(V) = \bigotimes_{d \geq 0} T^d(V)$ where $T^0(V)$ is the trivial module and $T^{d+1}(V) = V \otimes T^d(V)$ for all $d \geq 0$.

The symmetric group $S_n$ acts naturally on $T^d(V)$ by place permutation, and its action commutes with that of $G$. We define the symmetric powers to be the covariants of this action, i.e.
\[S^d(V) = \frac{T^d(V)}{\langle \sigma(v)-v: v \in T^d(V), \sigma \in S_n\rangle},\]
with induced $G$-action. Denote the image of $x_{i_1} \otimes \ldots \otimes x_{i_n}$ under this quotient by $x_{i_1} \cdots x_{i_n}$. Then the induced $G$-action is given by
\[g(x_{i_1} \cdots x_{i_n}) = (gx_{i_1}) \cdots g(x_{i_n})\] and the set $\{x_{i_1} \cdots x_{i_n}: 1 \leq i_1 \leq \cdots \leq i_n \leq d\}$ is a basis for $S^d(V)$; in particular the dimension of $S^d(V)$ is $\binom{n+d}{d}$. We note that $x_ix_j = x_jx_i$ for all $1 \leq i,j \leq n$.

Notice that by basic representation theory we have
\[D^d(V):= T^d(V)^{S_n} \cong S^d(V)^*.\] The modules $D^d(V)$ are sometimes called divided powers, but in the present article we will not use this term and these modules will be denoted as $S^d(V)^*$ rather than $D^d(V)$. 

Finally we define the exterior powers
\[\Lambda^d(V) = \frac{T^d(V)}{I},\] where $I$ is the 2-sided ideal of $T^d(V)$ generated by all tensors of the form $v \otimes v, v \in V$, with induced $G$-action. Denote the image of  $x_{i_1} \otimes \ldots \otimes x_{i_n}$ under this quotient by $x_{i_1} \wedge \cdots \wedge x_{i_n}$. Then the induced $G$-action is given by
\[g(x_{i_1} \wedge \cdots \wedge x_{i_n}) = g(x_{i_1}) \wedge \cdots \wedge g(x_{i_n})\] and the set $\{x_{i_1} \wedge \cdots \wedge x_{i_n}: 1 \leq i_1 < \cdots < i_n \leq d\}$ is a basis for $\Lambda^d(V)$; in particular the dimension of $\Lambda^d(V)$ is $\binom{n}{d}$. We note that $x_i \wedge x_j = -x_j \wedge x_i$ for all $1 \leq i,j \leq n$ and $x_i \wedge x_i = 0$. We also note that $\Lambda^d(V) = 0$ for $d>n$ and $\Lambda^n(V)$ is a one-dimensional module on which $G$ acts via the determinant character. Further, the natural pairing $\Lambda^{n-d}(V) \times \Lambda^d(V) \rightarrow \Lambda^n(V)$ induces an isomorphism
\[ \Lambda^{n-d}(V) \cong \Lambda^d(V)^*.\]

\subsection{Projective modules and the Heller operator}

In this section we let $G$ be an arbitrary group of order divisible by $p$. Recall that a $\kk G$-module $P$ is said to be projective if every exact sequence of $\kk G$-modules
\begin{equation}\label{Pproj} 0 \longrightarrow X \longrightarrow Y \longrightarrow P \longrightarrow 0\end{equation}
splits. It is said to be injective if every exact sequence of $\kk G$-modules
\[0 \longrightarrow P \longrightarrow Y \longrightarrow X \longrightarrow 0\]
splits. One can show that a $\kk G$-module is injective if and only if it is surjective.

Let $M$ be another $\kk G$-module. The $\kk G$-module is said to be {\bf projective relative to $M$} of the exact sequence \eqref{Pproj} splits whenever the tensored exact sequence
\begin{equation}\label{PXrelproj} 0 \longrightarrow M \otimes X \longrightarrow M \otimes Y\longrightarrow M \otimes P \longrightarrow 0\end{equation}
splits. Injective relative to $M$ is defined analogously and one can show that the two concepts are equivalent. Consider the special case where $M = \sum_{H \in \chi} \kk_H \uparrow^G_H$ where $\kk_H$ denotes the trivial indecomposable $\kk H$-module and $\chi$ is a set of subgroups of $G$. Then as a consequence of Frobenius' reciprocity theorem, \eqref{PXrelproj} splits if and only if \eqref{Pproj} splits on restriction to each $H \in \chi$. In this situation $P$ is usually said to be projective relative to $\chi$ rather than relative to $M$. Note that $P$ is then projective relative to the trivial subgroup of $G$ if and only if it is projective, and every $\kk G$-module is projective relative to $G$ itself.

The following omnibus Lemma sums up what we need to know about relative projectivity:

\begin{Lemma}\label{higman}
Let $G$ be a finite group of order divisible by $p$,  and $P,M$  $\kk G$-modules. Then the following are equivalent:
\begin{enumerate}
\item[(i)] $P$ is projective relative to $M$;
\item[(ii)] $P$ is injective relative to $M$;
\item[(iii)] $P$ is a direct summand of $X \otimes M$ for some $\kk G$-module $X$.
\end{enumerate}
\end{Lemma}
 
One may construct a category whose objects are $\kk G$-modules, with the set of morphisms between objects $X$ and $Y$ given by
\[\mathrm{Hom}_{\kk G, M}(X,Y) = \Hom_{\kk G}(X,Y)/\mathrm{PHom}_{\kk G}(X,Y)\] where $\mathrm{PHom}_{\kk G}(X,Y)$ denotes the linear space of all homomorphisms factoring through a module which is projective relative to $M$. In this category objects $X, Y$ are isomorphic if and only if there exist modules $P, Q$ which are projective relative to $M$ such that $X \oplus P \cong Y \oplus Q$. We write $P \cong_M Q$ for isomorphisms in this category. This is called the {\bf stable module category relative to $M$}. In case $M = \sum_{H \in \chi} \kk_H \uparrow^G_H$ we call it the stable module category relative to $\chi$, and if $\chi$ consists of the trivial subgroup simply the stable module category. The category has a triangulated structure, which is well-known in the case of the stable module category. However, in this article we will not make use of that structure.

Given a $\kk G$-module $M$, we can find a projective module $P_0$ and a surjective $\kk G$-morphism
\[\pi_0: P_0 \rightarrow M.\]
If we choose $P_0$ and $\pi_0$ so that $P_0$ has smallest possible dimension, then this pair is unique, and known as the projective cover of $M$. The kernel of $\pi_0$ is denoted $\Omega(M).$ This is known as the Heller shift of $M$. $\Omega(-)$ can be viewed as an operation on the set of $\kk G$-modules which takes indecomposable modules to indecomposable modules. The construction can be iterated, defining $\Omega^{i+1}(M) = \Omega(\Omega^{i}(M))$ for any $i \in \N$. Dually, 
given a $\kk G$-module $M$, we can find an injective (projective) module $I_0$ and a injective $\kk G$-morphism
\[\tau_0: M \rightarrow I_0.\]
If we choose $I_0$ and $\tau_0$ so that $I_0$ has smallest possible dimension, then this pair is unique, and known as the injective hull of $M$. The cokernel of $\tau_0$ is denoted $\Omega^{-1}(M).$ This construction can be iterated, defining $\Omega^{-(i+1)}(M) = \Omega^{-1}(\Omega^{-i}(M))$ for any $i \in \N$.
We also define $\Omega^0(M)$ to be the largest non-projective direct summand of $M$. With these conventions, we have the following proposition summarizing useful properties of the functors $\Omega^i$:

\begin{Proposition}\label{omegaomnibus} Let $M_1, M_2$ be $\kk G$-modules without projective summands, and $i, j$ nonzero integers. Then:

\begin{enumerate}
\item[(i)] $\Omega^i(M_1 \oplus M_2) \cong \Omega^i(M_1) \oplus \Omega^i(M_2)$;
\item[(ii)] $\Omega^i(M)^* \cong \Omega^{-i}(M^*)$;
\item[(iii)] $M \cong_{1} \Omega(\Omega^{-1}(M))  \cong_{1}  \Omega^{-1}(\Omega(M))$.
\item[(iv)] $\Omega(M_1) \otimes M_2 \cong_{1} \Omega(M_1 \otimes M_2)$.
\end{enumerate}
\end{Proposition}

A version of the above also exists for relative projectivity and for projectivity relative to a module, but we will not make use of these; see for instance \cite{ElmerRelCoh} and \cite{Lassueur} for precise statements.

\subsection{Representations of cyclic groups of order $p$}

In this section let $G$ denote a cyclic group of order $p$, and $\kk$ a field of characteristic $p$. The representation theory of $G$ is easy to understand, and in particular not wild. There exists exactly one indecomposable representation $V_i$ for each dimension $i<p$. The action of a generator of $G$ on $V_i$ is given by left-multiplication by a single Jordan block of size $i$. Each $V_i$ has fixed-point space $V_i^G$ with dimension 1. So, one may determine the number of indecomposable summands in an arbitrary representation of $G$ by simply computing the dimension of its fixed point space, and knowing the dimensions of its indecomposable summands is tantamount to knowing its complete decomposition.

It is now immediately obvious that
\begin{equation}\label{dual} V_i^* \cong V_i
\end{equation}
since a representation of any group is indecomposable if and only if its dual is. Further, let $v_1, v_2, \ldots, v_i$ be a basis of $V_i$ with respect to which the action of $G$ is as described above. For any $1 \leq j \leq i$, the elements $v_j, v_{j+1}, \ldots, v_i$ span a submodule of $V_i$ isomorphic to $V_{i-j+1}$, and the quotient with respect to this submodule is isomorphic to $V_{j-1}$ (here $V_0$ denotes the zero module).  In particular, there is a surjective homomorphism $\pi: V_p \rightarrow V_i$ for any $i \leq p$ whose kernel is a submodule of $V_p$ isomorphic to $V_{p-i}$.  Noting that $V_p$ is the unique projective indecomposable module for this group,   this shows that
 \begin{equation}\label{heller} \Omega(V_i) \cong V_{p-i}
\end{equation}
for any $1 \leq i \leq p$. In particular, the modules $V_i$ with $1 \leq i<p$ are periodic with period 2.

It is also easy to see that
\begin{equation}\label{symV} S^{i-1}(V_2) \cong V_i
\end{equation}
 for any $1 \leq i \leq p$.

Tensor products of the modules $V_i$ were first considered by Almkvist and Fossum \cite{AlmkvistFossum}. The fundamental result is
\begin{equation}\label{tensorbyV2} V_i \otimes V_2 \cong \left\{ \begin{array}{lr} V_{i+1} \oplus V_{i-1}&  i<p\\ V_p \oplus V_p & i=p \end{array} \right.  
\end{equation}
See \cite[Lemma~2.3.1]{BensonVecBook} for a straightforward proof. 

Using this result and induction one can prove that provided $i+j \leq p$, we have
\begin{equation}\label{tensorbyVj} V_i \otimes V_j \cong \bigoplus_{l=1}^{\min(i,j)} V_{j+i-(2l-1)}\end{equation} and
\begin{equation}\label{tensorbyVp-j} V_{p-i} \otimes V_{p-j} \cong (p-i-j) V_p \oplus (V_i \otimes V_j).\end{equation} 
Together these formulae tell us the direct sum decomposition of $V_i \otimes V_j$ for any $i,j \leq p$.

Almkvist and Fossum also considered symmetric and exterior powers. They devised a calculus by which one could compute decompositions of $S^d(V)$ and $\Lambda^d(V)$ for any $d$ and for any representation $V$ of $G$. The full theory is quite involved; the interested reader should consult \cite[Section~2.8]{BensonVecBook}. Some easily stated highlights are as follows:

\begin{Proposition}\label{projectivity} $S^d(V_{i+1})$ is projective whenever $d+i \geq p$.
\end{Proposition}  

\begin{Proposition}\label{periodicity} $S^d(V_{i+1}) \cong S^{d'}(V_{i+1})$ modulo projective summands, where $d'$ is the remainder when $d$ is divided by $p$;
\end{Proposition}  

\begin{Proposition}\label{symmetryexterior} $S^d(V_{i+1}) \cong S^i(V_{d+1})$ for all $d < p$, and $S^d(V_{i+1})  \cong \Lambda^d(V_{i+d})$ if in addition $d+i \leq p$.
\end{Proposition}  

\subsection{Known results for representations of $\SL_2(\kk)$}

In this subsection we consider the group $H = \SL_2(\kk)$ and let $E$ be the natural 2-dimensional $\kk H$-module. Assume $\kk$ is algebraically closed and denote $E_i = S^{i-1}(E)^*$. 

The modules $E_i$ for $i<p$ are irreducible and self-dual. Moreover, a complete classification of irreducible modules for $\SL_2(\kk)$ is given by the modules $L(\lambda): \lambda \in \N$ where $L(\lambda) = \bigotimes_{i=1}^k F^i(E_{\lambda_i})$ where $F$ denotes the Frobenius twist (i.e. $g \in G$ acts as $F(g)$ on
$F(E_i)$, where $F(g)$ is obtained from $g$ by raising each element of the matrix to the power $p$) and $\lambda = \lambda_k p^k + \ldots + \lambda_0$ is the $p$-adic decomposition of $\lambda$. For this reason, these modules are well-studied, and it is known (see for instance \cite[Theorem~3.1]{DotyHenke}) that $E_i \otimes E_2 \cong E_{i-1} \oplus E_{i+1}$. Now notice that for $i < p$ the $\kk G$-modules $V_i$ defined in the introduction are simply the restrictions of $E_i$ to some subgroup of the subgroup of upper triangular matrices in $H$. Since tensor products commute with restriction, we may obtain Theorem \ref{tensorformula} immediately in the special case $i<p$. However, the modules $E_i$ for $i \geq p$ are not so well-studied, and the rest of Theorem \ref{tensorformula} is new.

More recently, Wildon and McDowell \cite{WildonMcDowell} studied symmetric and exterior powers of $E$. In particular, they showed that $E_i \cong E_i^*$ if and only if $i<p$ or $i=p^k$ for some $k$. We may deduce that:
\begin{Proposition}\label{self-dual}
Suppose $i<p$ or $i=p^k \leq q$. Then $V_i \cong V_i^*$.
 \end{Proposition}
Wildon and McDowell also described isomorphisms
 \begin{equation}\label{hermite} S^i(S^{d}(E))^* \cong S^d(S^i(E)^*)
\end{equation} and \[S^i(S^{d}(E))^* \cong \Lambda^i(S^{i+d-1}(E)^*)\] for all $i$ and $d$. Combining the two gives us an isomorphism $S^d(E_{i+1}) \cong \Lambda^d(E_{d+i})$. 
Since exterior and symmetric powers commute with restriction we obtain a generalisation of Proposition \ref{symmetryexterior}:
\begin{Proposition}\label{wildon}
Let $d,i>0$. Then  $S^d(V_{i+1}) \cong S^i(V_{d+1}^*)^*$ for all $d,i < q$, and $S^d(V_{i+1})  \cong \Lambda^d(V_{i+d})$ if in addition $d+i \leq q$.
\end{Proposition}
They also show that the isomorphism \eqref{hermite} above is the unique modular analogue of classical Hermite reciprocity, in the sense that versions with stars in all other possible positions do not hold. In particular it is not true that $S^d(E_{i+1}) \cong S^i(E_{d+1})$ for all $d, i$. This suggests that if is likely not true that $S^d(V_{i+1}) \cong S^i(V_{d+1})$ for all $i,d<q$. In fact this is easy to see: take $q=4$, $i=2$ and $d=1$; then $S^d(V_{i+1}) = V_3$ and $S^2(V_2) = V_3^*$. One easily computes $\dim(V_3)^G=1$ and $\dim(V_3^*)=2$ so these modules cannot be isomorphic.

\section{Main results: elementary abelian $p$-groups}

In this section $G$ denotes an elementary abelian $p$-group of order $q:=p^n$, and $\kk$ a field of characteristic $p$.
We begin by defining bases for some of the modules were are interested in. Let $W$ be a faithful indecomposable $\kk G$-module of dimension 2 and let $\langle X,Y \rangle$ be a basis with respect to which the action of $G$ is as described in the introduction. Let $m>0$. A basis for $S^m(W)$ is given by $a_0, \ldots, a_m$ where $a_i = X^iY^{m-i}$. We identify $G$ with the appropriate subgroup $E_W$ of $\kk$. The action of $\alpha \in G$ on $S^m(W)$ is then given by the formula

\begin{equation}\label{action} \alpha \cdot a_i = \sum_{j=0}^i \binom{i}{j} \alpha^j a_{i-j}
\end{equation}

Note the formula is independent of $m$, and we have a chain of inclusions
\[S^0(W) \subset S^1(W) \subset S^2(W) \subset \ldots \]

Now let $x_0,\ldots, x_m$ be the basis of $S^m(W)^*$ dual to $a_0,\ldots, a_m$. The action of $G$ on $S^m(W)^*$ is given by the formula

\begin{equation}\label{dualaction} \alpha \cdot x_i  = \sum_{j=0}^{m-i}\binom{i+j}{i}(-\alpha)^j x_{i+j}.
\end{equation}

$S^m(W)^*$ and $S^m(W)$ are indecomposable for all $m<q$, and projective if $m=q-1$. We refer to \cite[Proposition~3.4]{ElmerSympowers} for a proof. We will have more to say in general about the modules $S^m(W)^*$ than their non-dualised counterparts, and we adopt the notation $V_m:= S^{m-1}(W)^*$, so that $V_m^* = S^{m-1}(W)$. 

The following results, which generalise Propositions \ref{projectivity} and \ref{periodicity}, are taken from \cite{ElmerSympowers}:

\begin{Proposition} Let $i<q, d<q$ and suppose that $i+d \geq q$. Then $S^d(V_{i+1})$ is projective relative to the set of subgroups of $G$ with order $\leq i$.
\end{Proposition}

\begin{Proposition} Let $i<q$, $d \in \N$ and let $d'$ denote the remainder when $d$ is divided by $q$. Then $S^{d}(V_{i+1}) \cong_{\chi} S^{d'}(V_{i+1})$, where $\chi$ is the set of subgroups of $G$ with order $\leq i$.
\end{Proposition}

In light of these results, we are given to wonder whether Equations \eqref{tensorbyV2}, \eqref{tensorbyVj}, \eqref{tensorbyVp-j}, and \eqref{heller}, may also be extended to the modules $V_1, \ldots, V_q$. Of course, since there are many $\kk G$-modules which are not isomorphic to any of these, we will usually need to give explicit isomorphisms, making the proofs considerably more involved.

\subsection{Heller Shift}

Dual to the inclusion $V_l^* \subset V_m^* $ one has the projection $\pi_{m,l}: V_m \rightarrow V_l$ given by setting $x_l, \ldots, x_{m-1}$ to zero. In particular this shows that the projective cover of each $V_m$ is $V_q$ and $\Omega(V_m) = \ker(\pi_{q,m})$. 

\begin{Proposition}\label{shift} We have, for $1 \leq m<q$, $$\Omega(V_m) \cong V_{q-m}^*.$$
\end{Proposition}

\begin{proof} Let $x_{m},x_{m+1}, \ldots, x_{q-1}$ be the basis of $\Omega(V_m) = \ker(\pi_{q,m})$. For each $k=0, \ldots, q-m-1$ define
\[y_k:= x_{q-1-k}.\] Then $y_0, \ldots, y_{q-m-1}$ is a basis for  $\Omega(V_m)$, and we have for any $\alpha \in G$
\begin{align*}
\alpha \cdot y_k &= \alpha \cdot x_{q-1-k}\\
&= \sum_{j=0}^k \binom{q-1-k+j}{j}(-\alpha)^j x_{q-1-k+j}\\
&= \sum_{j=0}^k  \binom{q-1-k+j}{j}(-\alpha)^j y_{k-j}\\
&= \sum_{j=0}^k (-1)^j  \binom{q-1-k+j}{j} \alpha^j y_{k-j}\\
&= \sum_{j=0}^k   \binom{k}{j} \alpha^j y_{k-j}\\
\end{align*}
by Lemma \ref{binomialmodp} below as required.
\end{proof}

In the above proof we used the following Lemma:

\begin{Lemma}\label{binomialmodp} Let $q=p^n$ be a prime power and let $i, j \in \N_0$ such that $j \leq k <q-1$. Then
\[\binom{q-1-k+j}{j} = (-1)^j\binom{k}{j} \mod p.\]
 \end{Lemma}

\begin{proof} 
The proof is by induction on $n$. If $n=1$ then we have in particular $j<p$ and
\begin{align*}
\binom{q-1-k+j}{j}  &= \frac{(q-1-k+j)(q-2-k+j) \cdots (q-k))}{j!}\\
				&=  \frac{(j-k-1)(j-k-2) \cdots (-k)}{j!} \mod p\\
				&= (-1)^j  \frac{(k-j+1)(k-j+2) \cdots (k)}{j!}\\
				&= (-1)^j\binom{k}{j}.
\end{align*} 

For the inductive step we use the following consequence of Lucas' Theorem: for $a,b \in \N_0$ write 
\[a=a_1 p + a_0, b= b_1p +b_0\]
 where $a_0, b_0$ are respectively the remainders when $a,b$ are divided by $p$. Then
\begin{equation}\label{Lucas} \binom{a}{b} = \binom{a_1}{b_1} \binom{a_0}{b_0} \mod p,\end{equation} where a binomial coefficient in which the bottom entry exceeds the top is interpreted as zero.

Now write
\[k= k_1 p + k_0, j = j_1p + j_0\] in the same fashion. We note that $q-1 = (p^{n-1}-1)p - (p-1)$, and that $j_1 \leq k_1<p^{n-1}$. There are two cases to consider: if $j_0 \leq k_0$ then the remainder when $q-1-k+j$ is divided by $p$ is $p-1-k_0+j_0$ and the quotient is $p^{n-1}-1-k_1+j_1$. Then by \eqref{Lucas} we have

\begin{align*}
\binom{q-1-k+j}{j} &= \binom{p^{n-1}-1-k_1+j_1}{j_1}  \binom{p-1-k_0+j_0}{j_0}\\
&=   (-1)^{j_1}\binom{k_1}{j_1} (-1)^{j_0} \binom{k_0}{j_0}\\
\intertext{by induction,}
&=   (-1)^{j_1p}\binom{k_1}{j_1} (-1)^{j_0} \binom{k_0}{j_0}\\
\intertext{since $(-1)^p = (-1)$ for all odd primes $p$, and if $p=2$ then $1=-1$,}
&=   (-1)^{j}\binom{k_1}{j_1} \binom{k_0}{j_0}\\
&=   (-1)^{j}\binom{k}{j} \\
\end{align*} as required. On the other hand if $j_0>k_0$ then the quotient and remainder when $q-1-k+j$ is divided by $p$ are $p^{n-1}-k_1+j_1$ and $j_0-k_0-1$ respectively, and 

\begin{align*}
\binom{q-1-k+j}{j} &= \binom{p^{n-1}-k_1+j_1}{j_1}  \binom{j_0-k_0-1}{j_0}\\
&=0 \ \text{because $j_0-k_0-1<j_0$, while}\\
(-1)^j \binom{k}{j} = (-1)^j  \binom{k_1}{j_1} \binom{k_0}{j_0}
&= 0 \ \text{because $k_0<j_0$.}
\end{align*}
\end{proof}

Notice that if we combine these propositions \ref{shift} and \ref{self-dual} we get
\[\Omega^{-1}(V_m) \cong \Omega^{-1}(V_m^*) \cong \Omega(V_m)^* \cong V_{q-m} \] for $1 \leq m<p$ or $m=p^k$, where we used Proposition \ref{omegaomnibus}(iii) for the second isomorphism.

\subsection{Tensor Products}
In this section we will prove Theorem \ref{tensorformula}. Since we already have this result in the cyclic case we assume that $n>1$ where $q=p^n$. Let $m<q$. A basis for $V_2 \otimes V_m$ is given by $\{x_i \otimes x_j: i=0,1; j=0,1, \ldots, m-1\}$. We write $y_i = x_0 \otimes x_i$ and $z_i = x_1 \otimes x_i$. The action of $\alpha \in G$ on this basis is given by

\begin{equation}
 \alpha \cdot y_i  = \sum_{j=0}^{m-1-i}\binom{i+j}{i}(-\alpha)^j (y_{i+j} - \alpha z_{i+j}).
\end{equation}
\begin{equation}
 \alpha \cdot z_i  = \sum_{j=0}^{m-1-i}\binom{i+j}{i}(-\alpha)^j z_{i+j}.
\end{equation}

Now set $w_i = y_i+z_{i-1}$ for all $i=0, \ldots, m$. We interpret $y_m, z_m$ and $y_{-1},z_{-1}$ as zero, so that $w_0 = y_0$ and $w_m = z_{m-1}$. We claim:

\begin{Lemma}
$\{w_0, \ldots, w_m\}$ spans a submodule $Y$ of $V_2 \otimes V_m$ isomorphic to $V_{m+1}$.
\end{Lemma}

\begin{proof}
We have, for any $\alpha \in G$, and $i=0, \ldots, m$,

\begin{align*} \alpha \cdot w_i &=  \sum_{j=0}^{m-1-i}\binom{i+j}{i}(-\alpha)^j (y_{i+j} - \alpha z_{i+j}) + \sum_{j=0}^{m-i}\binom{i-1+j}{i-1}(-\alpha)^j z_{i-1+j}.\\
 &=  \sum_{j=0}^{m-1-i}\binom{i+j}{i}(-\alpha)^j y_{i+j} +  \sum_{j=0}^{m-1-i}\binom{i+j}{i}(-\alpha)^{j+1} z_{i+j} + \sum_{j=0}^{m-i}\binom{i-1+j}{i-1}(-\alpha)^j z_{i-1+j}\\
 &=  \sum_{j=0}^{m-1-i}\binom{i+j}{i}(-\alpha)^j y_{i+j} +  \sum_{j=1}^{m-i}\binom{i+j-1}{i}(-\alpha)^{j}  z_{i+j-1} + \sum_{j=0}^{m-i}\binom{i-1+j}{i-1}(-\alpha)^j z_{i-1+j}\\
 &=  \sum_{j=0}^{m-1-i}\binom{i+j}{i}(-\alpha)^j y_{i+j} +  \sum_{j=0}^{m-i}\binom{i+j-1}{i}(-\alpha)^{j}  z_{i+j-1} + \sum_{j=0}^{m-i}\binom{i-1+j}{i-1}(-\alpha)^j z_{i-1+j}\\
\intertext{because the middle term with $j=0$ has a binomial coefficient $\binom{i-1}{i} = 0$}\\
 &=  \sum_{j=0}^{m-1-i}\binom{i+j}{i}(-\alpha)^j y_{i+j} +  \sum_{j=0}^{m-i}\binom{i+j}{i}(-\alpha)^{j}  z_{i+j-1}\\
\intertext{by Pascal's Identity}
&=  \sum_{j=0}^{m-i}\binom{i+j}{i}(-\alpha)^j y_{i+j} +  \sum_{j=0}^{m-i}\binom{i+j}{i}(-\alpha)^{j}  z_{i+j-1}\\
\intertext{because $y_m=0$}
&= \sum_{j=0}^{m-i}\binom{i+j}{i}(-\alpha)^j w_{i+j}
\end{align*}
as required.
\end{proof}

Now Theorem \ref{tensorformula}(b) follows from a result of Benson and Carlson \cite[Theorem~2.1]{BensonCarlson}: all indecomposable summands of $V_2 \otimes V_p$ must have dimension divisible by $p$. Since we've proven $V_2 \otimes V_p$ has a submodule of dimension $p+1$, $V_2 \otimes V_p$ must indeed be indecomposable.

For part (a) we need a further Lemma:  let $v_i = (i+1)y_{i+1} + (i+1-m)z_i$ for all $i=0, \ldots, m-2$. 

 \begin{Lemma}
$\{v_0, \ldots, v_{m-2}\}$ spans a submodule $X$ of $V_2 \otimes V_m$ isomorphic to $V_{m-1}$.
\end{Lemma}

\begin{proof}
We have, for any $\alpha \in G$, and $i=0, \ldots, m-2$,

\begin{align*} \alpha \cdot v_i &=  (i+1)\alpha \cdot y_{i+1} + (i+1-m) \alpha \cdot z_i\\
&=  \sum_{j=0}^{m-i-2} (i+1)\binom{i+1+j}{j}(-\alpha)^j (y_{i+j+1} - \alpha z_{i+j+1})  + \sum_{j=0}^{m-i-1}(i+1-m)\binom{i+j}{j} (-\alpha)^j  z_{i+j}\\
&=  \sum_{j=0}^{m-i-2} (i+1)\binom{i+1+j}{j}(-\alpha)^j y_{i+j+1} + \sum_{j=0}^{m-i-2} (i+1)\binom{i+1+j}{j}(-\alpha)^{j+1} z_{i+j+1} \\ +& \sum_{j=0}^{m-i-1}(i+1-m)\binom{i+j}{j} (-\alpha)^j  z_{i+j}\\
&=   \sum_{j=0}^{m-i-2} (i+j+1)\binom{i+j}{j}(-\alpha)^j y_{i+j+1} + \sum_{j=0}^{m-i-2}  (i+j+1)\binom{i+j}{j} (-\alpha)^{j+1} z_{i+j+1} \\ +& \sum_{j=0}^{m-i-1}(i+1-m)\binom{i+j}{j} (-\alpha)^j  z_{i+j}\\
\intertext{By \eqref{criticalobs} below, with $k=i+j$ and $l=i+1$}\\
&=   \sum_{j=0}^{m-i-2} (i+j+1)\binom{i+j}{j}(-\alpha)^j y_{i+j+1} + \sum_{j=1}^{m-i-1}  (i+j)\binom{i+j-1}{j-1} (-\alpha)^{j} z_{i+j} \\ +& \sum_{j=0}^{m-i-1}(i+1-m)\binom{i+j}{j} (-\alpha)^j  z_{i+j}\\
&=   \sum_{j=0}^{m-i-2} (i+j+1)\binom{i+j}{j}(-\alpha)^j y_{i+j+1} + \sum_{j=1}^{m-i-1}  j\binom{i+j}{j} (-\alpha)^{j} z_{i+j} \\ +& \sum_{j=0}^{m-i-1}(i+1-m)\binom{i+j}{j} (-\alpha)^j  z_{i+j}\\
\intertext{By \eqref{criticalobs} once more, this time with $k=i+j$ and $l=j$}\\
&=   \sum_{j=0}^{m-i-2} (i+j+1)\binom{i+j}{j}(-\alpha)^j y_{i+j+1} + \sum_{j=1}^{m-i-1}  (i+j-m+1) \binom{i+j}{j}(-\alpha)^j z_{i+j}+(i-m+1)z_{i}\\
&=   \sum_{j=0}^{m-i-2} \binom{i+j}{j}(-\alpha)^j ((i+j+1)y_{i+j+1} + (i+j-m+1)z_{i+j})\\
&=   \sum_{j=0}^{m-i-2} \binom{i+j}{j}(-\alpha)^j  v_{i+j}
\end{align*}
as required.
\end{proof}

In the proof above we used the following well-known result on binomial coefficients:  let $0 \leq l<k$. Then
\begin{equation}\label{criticalobs}l\binom{k}{l} = k \binom{k-1}{l-1}\end{equation}

We are now in a position to prove Theorem \ref{tensorformula}(a). 

\begin{proof}[Proof of Theorem \ref{tensorformula}(a)] We have shown that, for all $m<q$, $V_2 \otimes V_m$ contains submodules $Y \cong V_{m+1}$ and $X \cong V_{m-1}$. We claim that, provided $p \not | m$, $V_2 \otimes V_m = X \oplus Y$. Notice that $\dim(X \oplus Y) = 2m =  \dim(V_2 \otimes V_m)$, so it is enough to show that $X \cap Y = \{0\}$.

Let us define the support, $\mathrm{Supp}(x)$, of an element of $V_2 \otimes V_m$ to be the set of basis elements of $V_2 \otimes V_m$ whose coefficient for $x$ is nontrivial. Then $\mathrm{Supp}(v_i) \cap \mathrm{Supp}(w_j) = \emptyset$ unless $i=j-1 \leq m-2$. It follows that if
\[x = \sum_{i=0}^{m-2} \lambda_i v_i = \sum_{j=0}^m \mu_j w_j\] then we must have 
\[\lambda_i v_i = \mu_{i+1} w_{i+1}\] for all $i = 0, \ldots, m-2$, and also $\mu_m = \mu_0 = 0$. Therefore
\[\lambda_i( (i+1)y_{i+1} + (i+1-m)z_i) = \mu_{i+1}(y_{i+1} + z_i)\] for all such $i$.
Comparing coefficients of $y_{i+1}$ and $z_i$ gives us
\[\lambda_i(i+1) = \mu_{i+1} = \lambda_i(i+1-m)\]
for all such $i$, from which we obtain $m \lambda_i = 0$. Since $m$ is not divisible by $p$ we must have $\lambda_i = 0$ for all $i=0, \ldots, i-2$ and we conclude $x=0$ as desired. This shows that, provided $m$ is not divisble by $p$,  $V_2 \otimes V_m = X \oplus Y$ and therefore 
\[V_2 \otimes V_m \cong  V_{m-1} \oplus V_{m+1}\] as required.
\end{proof}

\begin{proof}[Proof of Corollary \ref{newtensorformula}] Write $\cong_M$ for isomorphism in the stable module category relative to $M$. The proof is by induction on $i$. The case $i=1$ is trivial and the case $i=2$ follows from Theorem \ref{tensorformula}(a). 

Let $i,j$ be as described and assume that \eqref{newtensorbyVj} holds. Consider $V_2 \otimes V_i \otimes V_j$. Assume $2<i<p-1$. On the one hand we have, by  Theorem \ref{tensorformula}(a) and inductive hypothesis,

\begin{align*}
V_2 \otimes V_i \otimes V_j &\cong (V_{i+1} \oplus V_{i-1}) \otimes V_j\\
&\cong_M V_{i+1} \otimes V_j \oplus  \bigoplus_{l=1}^{\min(i-1,j')} V_{j+i-2l}.
\end{align*}

On the other hand by induction we have
\[
V_2 \otimes V_i \otimes V_j  \cong_M V_2 \otimes  \bigoplus_{l=1}^{\min(i,j')} V_{j+i+1-2l}.\]

Note that since $l \leq \min(i,j')$ and $i+j' \leq p$, we have $rp<j+i+1-2l<rp+p$ for all $l$. In particular $j+i+1-2l$ is not divisible by $p$ for any $l$, and we may use Theorem \ref{tensorformula}(a) to evaluate:

\[ V_2 \otimes V_i \otimes V_j  \cong_M     \bigoplus_{l=1}^{\min(i,j')} V_{j+i+2-2l} \oplus \bigoplus_{l=1}^{\min(i,j')} V_{j+i-2l}.\]

There are three cases to consider here: if $i>j'$, then $i-1 \geq j'$ and $\min(i+1,j') = \min(i,j') = \min(i-1,j') = j'$. Therefore
\[ V_2 \otimes V_i \otimes V_j  \cong_M \bigoplus_{l=1}^{\min(i+1,j')} V_{j+i+2-2l} \oplus \bigoplus_{l=1}^{\min(i-1,j')} V_{j+i-2l}\]
from which we obtain, by the Krull-Schmidt theorem, \begin{equation}\label{answer} V_{i+1} \otimes V_j \cong_M \bigoplus_{l=1}^{\min(i+1,j')} V_{j+i+1-(2l-1)}\end{equation} as required.
Otherwise, if $i < j'$ then $\min(i+1,j') = \min(i,j')+1$, $\min(i-1,j')=\min(i,j')-1$ and

\begin{align*}
 V_2 \otimes V_i \otimes V_j  &\cong_M \bigoplus_{l=1}^{\min(i,j')} V_{j+i+2-2l} \oplus \bigoplus_{l=1}^{\min(i,j')} V_{j+i-2l}\\
&\cong \bigoplus_{l=1}^{\min(i,j')} (V_{j+i+2-2l}) \oplus V_{j-i} \oplus \bigoplus_{l=1}^{\min(i-1,j')} V_{j+(i-1)+1-2l}\\
&\cong \bigoplus_{l=1}^{\min(i+1,j')} V_{j+(i+1)+1-2l} \oplus \bigoplus_{l=1}^{\min(i-1,j')} V_{j+(i-1)+1-2l} \\
\end{align*}
from which we obtain again \eqref{answer} as required. Finally if $i=j'$ then we note that $j+i-2\min(i,j') = j-j' = rp$, so that $V_{j+i-2\min(i,j')}$ is projective relative to $M$, and therefore 

\begin{align*}
V_2 \otimes V_i \otimes V_j  &\cong_M \bigoplus_{l=1}^{\min(i,j')} V_{j+i+2-2l} \oplus \bigoplus_{l=1}^{\min(i,j')-1} V_{j+(i-1)+1-2l}.
\intertext{Then since $\min(i+1,j')=j'$ and $\min(i-1,j') = i-1$ we have}
V_2 \otimes V_i \otimes V_j &\cong_M \bigoplus_{l=1}^{\min(i+1,j')} V_{j+(i+1)+1-2l} \oplus \bigoplus_{l=1}^{\min(i-1,j')} V_{j+(i-1)+1-2l}.
\end{align*}
from which we obtain again \eqref{answer} as required.
\end{proof}

\begin{proof}[Proof of Corollary \ref{newtensorbyVp-j}]
The proof is by induction on $i$. For the case $i=1$, we work by induction on $j'$. If $j' = 0$ it is clear that both sides are projective relative to $M$. If $j'=1$, we may deduce this immediately from Corollary \ref{newtensorformula}. Hence, we assume $j'>1$. Consider $V_{p-1} \otimes V_2 \otimes V_{j-1}$. On the one hand we have
\begin{align*} V_{p-1} \otimes V_2 \otimes V_{j-1} &\cong V_{p-1} \otimes V_j \oplus V_{p-1} \otimes V_{j-2}\\
\intertext{and by induction this is}
 &\cong_M V_{p-1} \otimes V_j \oplus V_{pr+p-j'+2}.\\
\intertext{On the other hand we have}
 V_{p-1} \otimes V_2 \otimes V_{j-1} &\cong_M V_{2} \otimes V_{p-1} \otimes V_{j-1}\\
&\cong_M V_2 \otimes V_{pr+p-j'+1}\\
&\cong  V_{pr+p-j'+2} \oplus V_{pr+p-j'}\\
\intertext{from which we deduce that}
V_{p-1} \otimes V_j &\cong_M  V_{pr+p-j'}
    \end{align*}
as required. This proves the result in the case $i=1$.
Now assume $i>1$. Consider $V_2 \otimes V_{p-i+1} \otimes V_j$. On the one hand we have
\begin{align*}
V_2 \otimes V_{p-i+1} \otimes V_j &\cong V_{p-i} \otimes V_j \oplus V_{p-i+2} \otimes V_j.\\
\intertext{By induction this is}
&\cong_M V_{p-i} \otimes V_j \oplus V_{i-2} \otimes V_{pr+p-j'}.\\
\intertext{On the other hand, by induction}
V_2 \otimes V_{p-i+1} \otimes V_j &\cong_M V_2 \otimes V_{i-1} \otimes V_{pr+p-j'}\\
&\cong V_i \otimes V_{pr+p-j'} \oplus V_{i-2} \otimes V_{pr+p-j'}.\\
\intertext{from which we deduce that}
V_{p-i} \otimes V_j &\cong_M V_i \otimes V_{pr+p-j'}
\end{align*}
as required.
\end{proof}

Using Corollaries \ref{newtensorformula} and \ref{newtensorbyVp-j}, one may decompose $V_i \otimes V_j$ for any $i<p$, $j<q$, $i+j \leq q$ up to the addition of summands projective to $M$. Decomposing these in the usual module category seems to be a very difficult problem, but working relative to $M$ enables some interesting results to be proved. We were unfortunately unable to prove the following generalisations of Corollaries \ref{newtensorformula} and \ref{newtensorbyVp-j}:

\begin{Conj}\label{conj} Let $i,j<q$ with $i+j \leq q$. Write $i=rp+i'$, $j=sp+j'$ with $i',j'<p$. Assume $i'+j' \leq p$. Then we have
\begin{enumerate} 
\item[(a)] $$V_i \otimes V_j \cong_M \bigoplus_{l=1}^{\min(i',j')} V_{i+j-(2l-1)}.$$
\item[(b)] $$V_{pr+p-i'} \otimes V_{ps+p-j'} \cong_M V_i \otimes V_j.$$
\end{enumerate}
\end{Conj}

By way of evidence for part (a), we offer the following result:

\begin{Lemma}\label{notenough} Suppose that $p>2$. Then we have $V_{p+1} \otimes V_{p+1} \cong_M V_{2p+1}$.
\end{Lemma}

\begin{proof}
We first note that since $p>2$ we have  $V_{p+1} \otimes V_{p+1} \cong S^2(V_{p+1}) \oplus \Lambda^2(V_{p+1})$. Now Proposition \ref{wildon} implies that $ \Lambda^2(V_{p+1}) \cong S^2(V_p)$. The latter is a direct summand of $V_p \otimes V_p$ which is projective relative to $M$. Therefore 
$V_{p+1} \otimes V_{p+1} \cong_M S^2(V_{p+1})$ and it is enough to show that $S^2(V_{p+1}) \cong_M V_{2p+1}$.

We claim that  $S^2(V_{p+1}) \cong S^2(V_{p-1}) \oplus V_{2p+1}$. We will exhibit a submodule of $V_{p+1}$ isomorphic to $V_{p-1}$. Let $x_0, x_1, \ldots, x_{p}$ be the basis of $V_{p+1}$ with respect to which the action of $\alpha \in G$ is given by the usual formula \eqref{dualaction}. For $i=0, \ldots p-2$ we set $y_i = (i+1)x_{i+1}$. Then we have
\begin{align*} \alpha \cdot y_i = (i+1) \alpha \cdot x_{i+1} &= (i+1) \sum_{k=0}^{p-1-i} \binom{i+k+1}{k}(-\alpha)^k x_{i+1+k}. \end{align*}
Note that the coefficient of $x_{p}$ on the right hand side is $\binom{p}{i+1} = 0 \mod p$, so the sum only needs to run to $p-2-i$. Thus 
\begin{align*}\alpha \cdot y_i &=  (i+1) \sum_{k=0}^{p-2-i} \binom{i+1+k}{k}(-\alpha)^k \frac{1}{i+1+k}y_{i+k} \\
&=  \sum_{k=0}^{p-2-i} \binom{i+k}{k}(-\alpha)^k y_{i+k} \end{align*}
as required.

It follows that $S^2(V_{p+1})$ contains a submodule $Y$ isomorphic to $S^2(V_{p-1})$, with basis $\{x_ix_j: i,j = 1, \ldots, p-1\}$.
Now we define, for each $i=0, \ldots, 2p$
\[z_i = (-1)^i \sum_{j+k=i}x_jx_k.\]

We claim that $z_0, z_1, \ldots, z_{2p}$ span a submodule of $S^2(V_{p+1})$ isomorphic to $V_{2p+1}$. 
We have, for all $i=0, \ldots, 2p$

\begin{align*}
\alpha \cdot z_i &= (-1)^i \sum_{j+k=i}(\alpha \cdot x_j)(\alpha \cdot x_k)\\
&= (-1)^i \sum_{j+k=i}\left(\sum_{t=0}^p \binom{t}{j} (-\alpha)^{t-j} x_t \right) \left(\sum_{s=0}^p \binom{s}{k} (-\alpha)^{s-k} x_s \right)\\
&= (-1)^i \sum_{j+k=i} \sum_{r=0}^{2p} \sum_{s+t=r}\left(\binom{t}{j} (-\alpha)^{t-j} x_t \right) \left(\binom{s}{k} (-\alpha)^{s-k} x_s \right)\\
&= (-1)^i \sum_{j+k=i} \sum_{r=0}^{2p} (-\alpha)^{r-i} \sum_{s+t=r} \binom{t}{j}\binom{s}{k}  x_s x_t\\
&= (-1)^i  \sum_{r=0}^{2p} (-\alpha)^{r-i} \sum_{s+t=r}\left(  \sum_{j+k=i} \binom{t}{j}\binom{s}{k} \right)  x_s x_t\\
\intertext{by Lemma \ref{morecombs} below}
&= (-1)^i  \sum_{r=0}^{2p} (-\alpha)^{r-i} \sum_{s+t=r} \binom{r}{i} x_s x_t\\
&=\sum_{r=0}^{2p} (-\alpha)^{r-i} \binom{r}{i} z_r
\end{align*}
as required.

This proves our claim. Now to finish the proof, observe that $Z \cap Y = \{0\}$, since $z_i$ contains a unique term $2x_0x_i$ for $i \leq p$ or $2x_{i-p}x_p$ for $i>p$, none of which are contained in any term in $Y$ (again we need $p>2$ here). Thus, since $\dim(Z)+\dim(Y) = 2p+1+\frac12 p(p-1) = \frac12 (p+1)(p+2) = \dim(S^2(V_{p+1}))$ we have $S^2(V_{p+1}) = Y \oplus Z \cong S^2(V_{p-1}) \oplus V_{2p+1}$.

Finally we claim that $S^2(V_{p-1}) \cong_M 0$. To see this, note that $S^2(V_{p-1})$ is a direct summand of $V_{p-1} \otimes V_{p-1}$. By Corollary \ref{newtensorformula} we have $V_{p-1} \otimes V_{p-1} \cong_M V_1$. But as $\dim( S^2(V_{p-1}) ) = \frac12 p(p-1) = 0 \mod p$, we cannot have 
$S^2(V_{p-1}) \cong_M V_1$ so we must have  $S^2(V_{p-1}) \cong_M 0$ as required. Alternatively we can use the isomorphism $S^2(V_{p-1}) \cong \Lambda^2(V_p)$, and the latter is a summand of $V_p \otimes V_p$. This completes the proof.
\end{proof}

\begin{Lemma}\label{morecombs}
For all $i,s,t$ we have  $\sum_{j+k=i} \binom{t}{j}\binom{s}{k} =  \binom{s+t}{i}$.
\end{Lemma}

\begin{proof}
The expression $ \binom{t}{j}\binom{s}{k} $ on the left hand is the number of ways of choosing $j$ things from a set of $t$ things and a further $k$ things from a set of $s$ things. Clearly if we sum this over all $j,k$ with $j+k=i$ we get the number of ways of choosing $i$ things from a set of $s+t$ things.
\end{proof}

The result above can be easily adapted to show that, for all $k<p$, $S^k(V_{p+1})$ contains a submodule isomorphic to $V_{pk+1}$. But it is not clear that this is a direct summand - the complement is not just $S^k(V_{p-1})$ as the dimension is too small. 

If Conjecture \ref{conj} were true we could use it, along with some standard facts about the Heller operator, to decompose $V_i \otimes V_j$ up to summands projective relative to $M$, for any $i,j<q$. The following example gives an idea of how this might work in general:

\begin{eg} Let $q=9$. Then the following table describes the decompositions of $V_i \otimes V_j$ modulo $V_3 \oplus V_6 \oplus V_9$ whenever $i,j<q$ and $i,j$ are not divisible by $p$:

\vspace{1cm}
\begin{center}
\begin{tabular}{c|cccccc}
 $\otimes$ & $V_1$ & $V_2$ & $V_4$ & $V_5$ & $V_7$ & $V_8$\\
\hline
$V_1$ & $V_1$ & $V_2$ & $V_4$ & $V_5$ & $V_7$ & $V_8$\\
$V_2$ & $V_2$ & $V_1$ & $V_5$ & $V_4$ & $V_8$ & $V_7$\\
$V_4$ & $V_4$ & $V_5$ & $V_7$ & $V_8$ & $\Omega^{-1}(V_5)$ & $\Omega^{-1}(V_4)$\\
$V_5$ & $V_5$ & $V_4$ & $V_8$ & $V_7$ & $\Omega^{-1}(V_4)$ & $\Omega^{-1}(V_5)$\\
$V_7$ & $V_7$ & $V_8$ & $\Omega^{-1}(V_5)$ & $\Omega^{-1}(V_4)$ & $\Omega^{-1}(V_8)$ & $\Omega^{-1}(V_7)$\\
$V_8$ & $V_8$ & $V_7$ & $\Omega^{-1}(V_4)$ & $\Omega^{-1}(V_5)$ & $\Omega^{-1}(V_7)$ & $\Omega^{-1}(V_8)$\\

\end{tabular}
\end{center}
\vspace{1cm}
In the table above, the tensor products involving $V_1$ are obvious, and those involving $V_2$ are a consequence of Theorem \ref{tensorformula}. The calculation $V_4 \otimes V_4 \cong_M V_7$ is a consequence of Lemma \ref{notenough}. Now we have
\[V_4 \otimes (V_4 \otimes V_2) \cong V_4 \otimes (V_3 \oplus V_5) \cong_M V_4 \otimes V_5\]
and
\[(V_4 \otimes V_4) \otimes V_2 \cong_M V_7 \otimes V_2 \cong V_8 \oplus V_6 \cong_M V_8.\]
from which we find $V_4 \otimes V_5 \cong V_8$. Similarly
\[V_5 \otimes (V_4 \otimes V_2) \cong V_5 \otimes (V_3 \oplus V_5) \cong_M V_5 \otimes V_5 \]
and
\[(V_5 \otimes V_4) \otimes V_2 \cong_M V_8 \otimes V_2 \cong V_7 \oplus V_9 \cong_M V_7\]
from which we find $V_5 \otimes V_5 \cong V_7$. The remainder of the table can be filled in using $V_8= \Omega^{-1}(V_1)$, $V_7 = \Omega^{-1}(V_2)$ and Propostion \ref{omegaomnibus}(iv).

\end{eg}

\bibliographystyle{plain}
\bibliography{MyBib.bib}

\end{document}